\documentclass[11pt]{article}
\usepackage{amsmath}
\usepackage{amssymb}
\usepackage{latexsym}
\usepackage{amsthm}
\usepackage{tabularx}
\usepackage{booktabs}
\usepackage{mathrsfs}
\usepackage{enumitem}
\usepackage{ytableau}
\usepackage{graphicx}
\usepackage{algorithm,algorithmic}
\usepackage[colorlinks,
            linkcolor=red,
            anchorcolor=blue,
            citecolor=green]{hyperref}
\usepackage{todonotes}

\usepackage[perpage,symbol*]{footmisc}
\usepackage{subfig}
\usepackage{mathtools}
\usepackage{stmaryrd}
\usepackage{caption}
\newcommand{\T}{\mathcal T}%Tableaux boisés
\newcommand{\s}{\sigma}

\renewcommand{\b}[1]{{#1}}
\renewcommand{\r}[1]{{\dot{#1}}}
\newcommand{\A}[1]{A_{a,b}(#1)}

\let\leq=\leqslant
\let\geq=\geqslant

\newtheorem{thm}{Theorem}[section]
\newtheorem{lem}[thm]{Lemma}
\newtheorem{prop}[thm]{Proposition}
\newtheorem{cor}[thm]{Corollary}
\newtheorem{conj}[thm]{Conjecture}

\newtheorem{defi}[thm]{Definition}

\allowdisplaybreaks

\begin{document}
\begin{center}
{\large \bf Enumeration of Corners in Tree-like Tableaux and a Conjectural ($a$,$b$)-analogue}
\end{center}

\begin{center}
Alice L.L. Gao$^{1,*}$, Emily X.L. Gao$^{2,*}$, Patxi Laborde-Zubieta$^{3,+}$ and Brian Y. Sun$^{4,\dag}$\footnotetext{$^\dag$Corresponding author.} \\[6pt]

$^{*}$Center for Combinatorics, LPMC-TJKLC\\
Nankai University, Tianjin 300071, P. R. China\\[6pt]

$^{+}$LaBRI - University of Bordeaux,\\
351, cours de la Libération F-33405 Talence cedex, France\\[6pt]

$^\dag$College of Mathematics and System Science,\\
XinJiang University, Urimqi, 830046, P. R. China\\[6pt]
Email: $^{1}${\tt gaolulublue@mail.nankai.edu.cn},
$^{2}${\tt gaoxueling@mail.nankai.edu.cn},\\
$^{3}${\tt plaborde@labri.fr},
$^{4}${\tt brianys1984@126.com}
\end{center}

\noindent\textbf{Abstract.} In this paper, we confirm a conjecture of Laborde-Zubieta on the enumeration of corners
in tree-like tableaux. Our proof is based on Aval, Boussicault and Nadeau's bijection between tree-like tableaux and permutation tableaux, and Corteel and Nadeau's bijection between permutation tableaux and permutations. This last bijection sends a corner in permutation tableaux to an ascent followed by a descent in permutations, this enables us to enumerate the number of corners in permutation tableaux, and thus to completely solve L.-Z.'s conjecture.
Moreover, we give a bijection between corners and runs of size 1 in permutations,
which gives an alternative proof of the enumeration of corners.
Finally, we introduce an ($a$,$b$)-analogue of this enumeration, and explain
the implications on the PASEP.

\noindent \emph{AMS Classification 2010:} 05A05, 05A19


\section{Introduction}

Tree-like tableaux were introduced by Aval, Boussicault and Nadeau in \cite{ABN}, they are related to many combinatorial objects such as permutations, permutation tableaux, non-ambiguous trees, alternative tableaux, Dyck tableaux and so on. For more information, we refer the reader to \cite{ABB,AB,ABN,N,SW,V}. The main objective of this paper is to prove a conjecture due to Laborde-Zubieta \cite{LZ}, which gives a formula to enumerate the number of corners in tree-like tableaux, we present two proofs of this conjecture. Our first proof uses the relationship between permutation tableaux and tree-like tableaux. The second one, is based on a bijection between corners and runs of size 1 in permutations. L.-Z. also conjectured a formula on the enumeration of corners in symmetric tree-like tableaux, which was proved in \cite{GGSY}. Note that Hetczenko and Lohss \cite{HL} independently confirmed these two conjectures via the enumeration of corners in permutation tableaux, while they used a different method to enumerate the corners in permutation tableaux.
We also give a conjecture about an ($a$,$b$)-analogue of the enumeration of corners.

First, let us recall some basic notions and tools about tree-like tableaux and their underlying Young diagrams.
For two nonnegative integers $n,k$ with $n\geq k,$ a \emph{$(k,n)$-diagram} $\mathcal{D}$ is a left-justified diagram of boxes in a $k\times (n-k)$ rectangle with $\lambda_i$ boxes in the $i$-th row, where $\lambda_1\geq \lambda_2\geq\cdots\geq\lambda_k\geq 0$.
The \emph{length} of a $(k,n)$-diagram $\mathcal{D}$ is defined to be $n$, i.e., the number of its rows plus the number of its columns.
Note that a $(k,n)$-diagram $\mathcal{D}$ may have empty rows and columns.
For example, Figure \ref{3-7-diagram} is a $(3,7)-$diagram.

\setlength{\unitlength}{1mm}
\begin{center}
\begin{figure}[H]
\begin{picture}(35,12)
\multiput(70,-1)(5,0){1}{\line(0,1){15}}
\multiput(75,4)(5,0){3}{\line(0,1){10}}
\multiput(70,14)(0,5){1}{\line(1,0){20}}
\multiput(70,4)(0,5){2}{\line(1,0){15}}
%\put(60.5,0){\small 8}\put(62,1){7}
%\put(67,1){6}\put(72,1){5}
%\put(75.7,5.5){4}\put(75.5,10){3}\put(77.5,10.5){2}\put(81.5,10.5){1}
%R\thicklines
%\put(85,14){\vector(-1,0){5}}
%\put(80,14){\vector(-1,0){5}}
%\put(75,14){\vector(0,-1){5}}
%\put(75,9){\vector(0,-1){5}}
%\multiput(75,4)(-5,0){3}{\vector(-1,0){5}}
%\put(60,4){\vector(0,-1){5}}
\end{picture}
\caption{A $(3,7)$-diagram. }\label{3-7-diagram}
\end{figure}
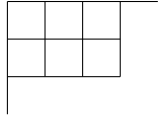
\end{center}

In fact, a  $(k,n)$-diagram is determined uniquely by its southeast border.
We call an edge a \emph{border edge} if it lies on the southeast border.
Moreover, a border edge is called a \emph{south step} if it goes down along the southeast border and a \emph{west step} otherwise.
We label border edges from 1 to $n$ starting from the top right corner.
An example of this labelling is given in the left part of Figure~\ref{4-8-diagram}.

Let $\mathcal{D}$ be a $(k,n)$-diagram, rows and columns of $\mathcal{D}$ will be labeled by integers in the following manner.
Since the border edges in the southeast border of $\mathcal{D}$ are labeled by $[n]=\{1,2,\cdots,n\}$,
if a south step is labeled by $i$, then the corresponding row is labeled with $i$
and if a west step is labeled by $j$, then the corresponding column is labeled with $j$,
as illustrated in the right part of Figure~\ref{4-8-diagram}.
We will say \emph{row $i$} or  \emph{column $j$} when we actually refer to the row with the label $i$ or the column with the label $j$.
The cell located at the intersection of row $i$
and column $j$ can be denoted by $(i,j)$.

\setlength{\unitlength}{1mm}
\begin{center}
\begin{figure}[H]
\begin{picture}(50,20)
\multiput(45,0)(5,0){1}{\line(0,1){20}}

\multiput(50,5)(5,0){3}{\line(0,1){15}}
\multiput(65,15)(5,0){1}{\line(0,1){5}}
\multiput(45,5)(0,5){2}{\line(1,0){15}}
\multiput(45,15)(0,5){2}{\line(1,0){20}}
\put(45.5,1){8}\put(60.5,7){4}\put(60.5,11){3}
\put(66,17){1}\put(47.5,2){7}
{\thicklines
\put(65,20){\vector(0,-1){5}}
\put(65,15){\vector(-1,0){5}}
\multiput(60,10)(0,5){2}{\vector(0,-1){5}}
\multiput(60,5)(-5,0){3}{\vector(-1,0){5}}
\put(45,5){\vector(0,-1){5}}
}
\put(52,2){6}\put(57,2){5}\put(62.5,12){2}
%%%%%%%%%%%%%%%%%%%%%%%%%%%%%%%%%%%%%%
\multiput(90,0)(5,0){1}{\line(0,1){20}}
\multiput(95,5)(5,0){3}{\line(0,1){15}}
\multiput(110,15)(5,0){1}{\line(0,1){5}}
\multiput(90,5)(0,5){2}{\line(1,0){15}}
\multiput(90,15)(0,5){2}{\line(1,0){20}}
\put(87.5,1){8}\put(87.5,6){4}\put(87.5,11){3}
\put(87.5,16){1}\put(91.5,21){7}
\put(96.5,21){6}\put(101.5,21){5}\put(106.5,21){2}
%%%%%%%%%%%%%%%%%%%%%%%%%%(4,8)-diagram-above
\end{picture}
\caption{
    A $(4,8)$-diagram with the labelling of its southeast border
    in bold (left) and the labelling of its rows and columns (right).
}
\label{4-8-diagram}
\end{figure}
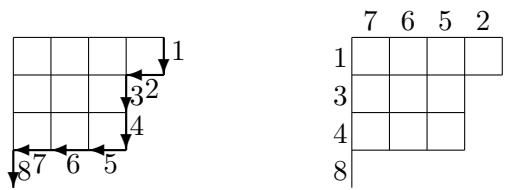
\end{center}

Suppose $\mathcal{D}$ is a $(k,n)$-diagram,
a \emph{tableau} $T$ with shape $\mathcal{D}$ is a certain filling of the cells in $\mathcal{D}$ with some symbols. The \emph{length} of $T$ is defined to be the length of $\mathcal{D}$.
We call the diagram obtained by removing all the symbols in a tableau $T$ the \emph{underlying diagram} of $T$, denoted by $\mathcal{D}=\mathcal{D}(T)$.

\begin{defi}\textnormal{
A cell in a $(k,n)$-diagram $\mathcal{D}$ is said to be a \emph{corner} if its bottom and right edges are both border edges of $\mathcal{D}$.}
\end{defi}

For example, the $(3,7)$-diagram in Figure \ref{3-7-diagram} has only one corner, i.e., $\{(3,4)\}$ and the $(4,8)$-diagram in Figure \ref{4-8-diagram}
has two corners, i.e., $\{(1,2),(4,5)\}$.

A cell of a tableau $T$ is called a \emph{corner} if this cell is a corner in the underlying diagram of $T$.
Moreover, it is called an \emph{occupied corner} if this cell is filled with a symbol in $T$. Otherwise, a corner
     cell is called a \emph{non-occupied corner}.
Denote by $c(T)$ the number of corners of a tableau $T$ and extend it to a set $X$ of tableaux as follows, $$c(X)=\sum_{T\in X}c(T)$$
     which means that $c(X)$ is the number of corners in $X$.

In what follows, we shall recall two kinds of tableaux,
which are the main combinatorial objects of our work, namely tree-like tableaux and permutation tableaux.

\begin{defi}\textnormal{
A \emph{tree-like tableau} is a filling of $(k,n)-$diagram (without empty rows or empty columns) with points inside some cells, such that the resulting diagram satisfies the following three rules,
\begin{itemize}
 \item[(1)] the top left cell of the diagram contains a point, called the \emph{root point};
 \item[(2)] for every non-root pointed cell $c$, there exists either a pointed cell above $c$ in the same column, or a pointed cell to its left in the same row, but not both;
 \item[(3)] every column and every row possess at least one pointed cell.
\end{itemize}}
\end{defi}

 The \emph{size} of a tree-like tableau is defined to be its number of points. It is not difficult to see that the length of a tree-like tableau equals to its size plus one \cite{ABN}. In the sequel, we denote by $\mathcal{T}_n$ the set of the tree-like tableaux of size $n$.
%Tree-like tableau was first introduced by Aval et al \cite{ABN}.
An example of a tree-like tableau of size 8 is shown in the following figure.

\begin{center}
\makeatletter\def\@captype{figure}\makeatother
\begin{minipage}{.075\textwidth}
\begin{ytableau}
\bullet &&  \bullet&\\
\bullet&\bullet& &\bullet\\
&&\bullet\\
&\bullet\\
\bullet
\end{ytableau}
\end{minipage}
\caption{A tree-like tableau of size $8$.}
\label{fig: tree-like example}
\end{center}

Permutation tableaux arose in the study of totally nonnegative Grassmanian, see Postnikov \cite{post}. There have been a lot of work on the subject in many different directions since it was formally introduced by Steingr\'{i}msson and Williams in \cite{SW}. See \cite{CK,CN,N,burstein,cw}. One of the main importance of permutation tableaux relies on the fact that they are in bijection with permutations \cite{SW,burstein,CK,CN}, and tree-like tableaux \cite{ABN,LZ}.

\begin{defi}\textnormal{
A \emph{permutation tableau}, is a $(k,n)$-diagram with no empty columns together with a 0,1-filling of the cells such that
\begin{itemize}
\item[(1)] each column has at least one 1;
\item[(2)] there is no 0 which has a 1 above it in the same column and a 1 to the left of it in the same row.
\end{itemize}}
\end{defi}

We denote by $\mathcal{PT}_n$ the set of permutation tableaux of length $n$.
%Different statistics on permutation tableaux were defined in \cite{SW} and we just list a few here. A 1 in a permutation tableau is a \emph{topmost} 1 if there is no 1 above it in the same column. A 0 in a permutation tableau is \emph{restricted} if there is a 1 above it in the same column and it is a \emph{rightmost restricted} 0 if it is a restricted 0 and there is no restricted 0 to its right in the same row. A row is called \emph{unrestricted} if it has no restricted 0.
%Permutation tableaux were originally introduced by Williams and Steingr\'{i}msson \cite{SW}.
An example of permutation tableau of length $8$ is given in Figure \ref{permutation tableau}.

\setlength{\unitlength}{1mm}
\begin{center}
\begin{figure}[H]
\begin{picture}(40,15)
\multiput(70,0)(5,0){1}{\line(0,1){20}}
\multiput(75,5)(5,0){3}{\line(0,1){15}}
\multiput(90,15)(5,0){1}{\line(0,1){5}}
\multiput(70,5)(0,5){2}{\line(1,0){15}}
\multiput(70,15)(0,5){2}{\line(1,0){20}}
\put(71.5,6.5){0}\put(76.5,6.5){0}\put(81.5,6.5){1}
\put(71.5,11.5){1}\put(76.5,11.5){1}\put(81.5,11.5){1}
\put(71.5,16.5){0}\put(76.5,16.5){1}\put(81.5,16.5){0}\put(86.5,16.5){1}
\end{picture}
\caption{A permutation tableau of length 8.}
\label{permutation tableau}
\end{figure}
\end{center}

A \emph{permutation} $\pi$ of length $n$ is a bijection from $[n]$ to $[n]$,
we use the notation $\pi_i:=\pi(i)$ for $1\leq i\leq n$. The group of permutations of length $n$ is denoted by $\mathfrak{S}_n$.
Given a permutation $\pi=\pi_1\cdots\pi_n\in \mathfrak{S}_n$
with the convention that $\pi_{n+1}=n+1$,
we consider the following non usual definitions
for ascents and descents, we say that $\pi_i$ is a \emph{descent}
if $\pi_i>\pi_{i+1}$ and call $\pi_i$ an \emph{ascent}
if $\pi_i<\pi_{i+1}$ for $1\leq i\leq n$.

In \cite{LZ}, L.-Z. studied the enumeration of occupied corners and obtained the following results.
\begin{prop}[\cite{LZ}]\label{prop-Z} The number of occupied corners in $\mathcal{T}_n$ is $n!$.
\end{prop}

Moreover, L.-Z. also considered the enumeration of corners in the set of tree-like tableaux of size $n$ and posed the following conjecture.
\begin{conj}\emph{(\cite[Conjecture 4.1]{LZ}).}\label{conj1}
The number of corners in the set of tree-like tableaux of size $n$ is $n!\times\frac{n+4}{6}$, i.e,
$$c(\mathcal{T}_n)=n!\times\frac{n+4}{6}.$$
\end{conj}
We will confirm this conjecture in the next section.

% by using the bijection between permutations and permutation tableaux built by Corteel and Nadeau \cite{CN}, and also the bijection between tree-like tableaux and permutation tableaux established by Aval et al \cite{ABN}.

\section{Proof of Conjecture \ref{conj1}}

In order to give an enumerative result for $c(\mathcal{T}_n)$, we first describe a bijection $\phi:\mathcal{T}_n \rightarrow \mathcal{PT}_n$, which is due to  Aval et al \cite{ABN}. Besides, we also need a bijection  $\varphi:\mathcal{PT}_n\rightarrow  \mathfrak{S}_n$, which was found by Corteel and Nadeau \cite{CN}.
%Simple introduction about these bijections %will be given for the sake of completeness %of our paper.
%Based on these two bijections, we will give a simple proof of Conjecture \ref{conj1} in the end.

\subsection{Bijection \texorpdfstring{$\phi$}{Lg} Between Tree-like Tableaux and Permutation Tableaux}

In this subsection, we exhibit the bijection $\phi$.
Here we just give a simple description of the bijection and omit the description of its inverse, and the reader can refer to \cite{ABN} for more details about the proof.

\begin{lem}[\cite{ABN}]\label{prop}
There exists a bijection $\phi:\mathcal{T}_n\rightarrow \mathcal{PT}_n$ such that for any tree-like tableau $T\in \mathcal{T}_n$, the underlying Young diagram of permutation tableau $PT=\phi(T)\in \mathcal{PT}_n$ is obtained from $\mathcal{D}$ by removing the first column of $\mathcal{D}$, where $\mathcal{D}=\mathcal{D}(T)$.
\end{lem}

Given a tree-like tableau $T$, firstly, change the topmost point $p$ in every column to 1. Secondly, transform every other point $q$ in $T$ to 0. Lastly, fill every empty cell with a 0 if it is in the left of 0 in the same row or in the above of 1 in the same column,  and fill the rest of empty cells with $1s$.  %changing every non-root point $p$ that is the topmost in its column to 1, and transforming every non-root point $p$ that is the rightmost one in its row to 0. Filling every empty cell with a 0 if it is in the left of 0 in the same row or in the above of 1 in the same column. Then filling the other empty cells with 1.
In the end,  simply delete the leftmost column of $T$. The resulting tableau we get is a permutation tableau, denoted by $\phi(T)$.
As an example of the bijection $\phi$, see  Figure \ref{treelikepermutation}.

\setlength{\unitlength}{1mm}
\begin{center}
\begin{figure}[H]
\begin{picture}(40,18)
\multiput(45,0)(5,0){2}{\line(0,1){20}}
\multiput(55,5)(5,0){3}{\line(0,1){15}}
\multiput(70,15)(5,0){1}{\line(0,1){5}}
\multiput(45,0)(0,5){2}{\line(1,0){5}}
\multiput(45,5)(0,5){2}{\line(1,0){20}}
\multiput(45,15)(0,5){2}{\line(1,0){25}}
\put(46.5,1.5){$\bullet$}\put(46.5,11.5){$\bullet$}
\put(46.5,16.5){$\bullet$}\put(51.5,11.5){$\bullet$}
\put(56.5,6.5){$\bullet$}\put(56.5,16.5){$\bullet$}
\put(61.5,11.5){$\bullet$}\put(66.5,16.5){$\bullet$}
%%%%%%%%%%%%%%%%%%%%%%%%%%%%%%%%%%%%%%%%%%%%%%%%%%%%%%
\put(80,10){$\longrightarrow$}
\put(82,12){$\phi$}
\multiput(95,0)(5,0){1}{\line(0,1){20}}
\multiput(100,5)(5,0){3}{\line(0,1){15}}
\multiput(115,15)(5,0){1}{\line(0,1){5}}
\multiput(95,5)(0,5){2}{\line(1,0){15}}
\multiput(95,15)(0,5){2}{\line(1,0){20}}
\put(96.5,6.5){0}\put(101.5,6.5){0}\put(106.5,6.5){1}
\put(96.5,11.5){1}\put(101.5,11.5){1}\put(106.5,11.5){1}
\put(96.5,16.5){0}\put(101.5,16.5){1}\put(106.5,16.5){0}\put(111.5,16.5){1}
%%%%%%%%%%%%%%%%%%%%%%%%%%%%%%%%%%%%%%%%%
\end{picture}
\caption{The bijection $\phi$ from a tree-like tableau to a permutation tableau.}
\label{treelikepermutation}
\end{figure}
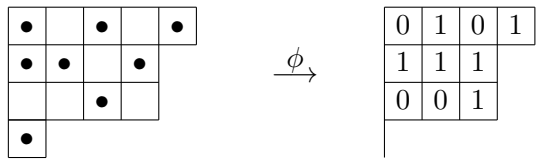
\end{center}

By Lemma \ref{prop}, we can see that for a tree-like tableau $T$, the corners of $T$ are closely related to the corners of the permutation tableau $\phi(T)$. As a matter of fact, we have the following result.
%Also, we can use $c(\mathcal{PT}_n)$ to denote the number of corners in the set of permutation tableaux of length $n$.
\begin{cor}
The number of corners in $\mathcal{T}_n$ equals to the number of corners in $\mathcal{PT}_n$ plus $|X_n|$, i.e.,
\begin{align*}\label{eq-1}
c(\mathcal{T}_n)=c(\mathcal{PT}_n)+|X_n|.
\end{align*}
where $X_n$ is the subset of $\mathcal{PT}_n$  such that the border edge labeled by $n$ in the underlying diagram of its element is a south step.
\end{cor}

\proof
For any tree-like tableau $T\in \mathcal{T}_n$ the underlying diagram $D$ of $T$ is of length $n+1$. It is easy to see that the border edge labeled by $n+1$ in $\mathcal{D}$ is always a west step since $\mathcal{D}$ has no empty rows. Let $PT$ denote the permutation tableau $\phi(T)$.

By Lemma \ref{prop}, we know that
the underlying diagram $\mathcal{D}'$ of $PT$ is obtained by simply removing the leftmost column of $\mathcal{D}$.
That is to say, all the border edges but the one labeled by $n+1$ in the southeast border of $\mathcal{D}$ are the same as the whole border edges of $\mathcal{D}'$ one by one.

Thus the number of corners in $\mathcal{D}'$ equals to the number of corners in $\mathcal{D}$ if the border edge labeled by $n$ in the southeast border of  $\mathcal{D}$ is a west step. In this case, the number of  corners in $PT$ is equal to $c(T)$. Otherwise, if the border edge labeled by $n$ in the southeast border of  $\mathcal{D}$ is a south step, one can see that the border edges labeled by $n$ and $n+1$ form a corner. So the number of corners in $\mathcal{D}'$ must be less one than  the corners in $\mathcal{D}$, this case implies that the number of corners in $PT$ is equal to $c(T)$ minus one.

Therefore, the number of corners in $\mathcal{PT}_n$ equals to $c(\mathcal{T}_n)$ minus $|X_n|$.
\qed

\subsection{A Bijection \texorpdfstring{$\varphi$}{Lg} from Permutation  Tableaux to Permutations}
To begin with, it is worthy to mention that there have been several bijections from Permutation Tableaux to Permutations up to now, see \cite{burstein, CN, CK,SW}. Here we only introduce the one  due to Corteel and Nadeau \cite{CN}. We list the following lemma without giving the proof, the reader can refer \cite{CN} for details.

\begin{lem}\emph{(\cite[Theorem 1. (1)]{CN}).}\label{lem:bij-PT-Per}
There exists a bijection $\varphi:\mathcal{PT}_n \rightarrow \mathfrak{S}_n $ such that for any permutation tableau $PT\in \mathcal{PT}_n$ and $1\leq i\leq n$,  $i$ is a label of a west step in $PT$ if and only if $i$ is a descent in $\pi$; and $i$ is a label of a row in $PT$ if and only if $i$ is an ascent in $\pi$, where $\pi=\varphi(PT)$.
\end{lem}

By Lemma \ref{lem:bij-PT-Per}, we can easily find the following result.

\begin{cor}\label{cor:bij-per-PT}
For a permutation tableau $PT\in \mathcal{PT}_n$ and $1\leq i < n$, the consecutive border edges labeled with $i$ and $i+1$ in the southeast border in its underlying diagram form a corner if and only if $i$ is an ascent and $i+1$ is a descent in $\pi=\varphi(PT)$.
\end{cor}

Therefore, we can proceed to enumerate the number of corners in $\mathcal{PT}_n$ by considering the enumeration of the pairs $(\pi,i)(1
\leq i<n)$ for which $\pi\in \mathfrak{S}_n$ and $i$ is an ascent and $i+1$ is a descent in $\pi$.
%Thus we can deduce the following theorem.
\begin{thm}\label{corner-number-A}
For positive integer $n$,
the number of corners in the set of permutation tableaux of length $n$ is counted by
\[c(\mathcal{PT}_n)=\left\{\begin{array}{cc}0,&\mathrm{~if~}n=1,\\
(n-1)!\times\frac{n^2+4n-6}{6},&\mathrm{~if~}n\geq2.\end{array}\right.\]
\end{thm}

\begin{proof}
For $n=1$, there is only one permutation tableau of length 1 and hence not any corner in this case.

Now let's consider the number of corners in permutation tableaux of length $n\geq 2.$
Combing with Corollary \ref{cor:bij-per-PT}, we can compute the number of corners in a permutation tableau $PT$ by the number of $i'$s in the permutation $\pi$ such that $i(1\leq i < n)$ is an ascent and $i+1$ is a descent in $\pi$, where $\pi=\varphi(PT)$. That is to say,
$$c(PT)=\sum_{i=1}^{n-1}\chi(a_i)$$
where $a_i$ means that $i$ is an ascent and $i+1$ is a descent in  $\pi=\varphi^{-1}(PT)$
and
\[\chi(a_i)=\left\{\begin{array}{cc} 1,& \mathrm{if~}a_i~\mathrm{is~ture,}\\
0,& ~\mathrm{otherwise.}\end{array}\right.\]

Recall that our aim is to count the number of corners in the set of permutation tableaux of length $n$, so with the preliminaries above we can proceed to count such corners
by

$$c(\mathcal{PT}_n)=\sum_{PT\in \mathcal{PT}_n}c(PT)
=\sum_{\pi\in \mathfrak{S}_n}\sum_{i=1}^{n-1}\chi(a_i)
=\sum_{i=1}^{n-1}\sum_{\pi\in \mathfrak{S}_n}\chi(a_i).$$
%where $a_i$ means $i$ is an ascent and $i+1$ %is a descent in $\pi$.

For $1\leq i < n$, let $B_i$ denote the set of permutations in $\mathfrak{S}_n$ such that $i$ is an ascent and $i+1$ is a descent in $\pi\in\mathfrak{S}_n$, and $|B_i|$ the cardinality of $B_i$. It is clearly that
$$|B_i|=\sum_{\pi\in \mathfrak{S}_n}\chi(a_i).$$
So it suffices to compute $|B_i|$ in order to count $c(\mathcal{PT}_n)$.

For any permutation $\pi=(\pi_1,\ldots,\pi_n)\in B_i$, suppose there exist $1\leq t_1,t_2 \leq n$ such that $\pi_{t_1}=i$ and $\pi_{t_2}=i+1$.
By the definition of ascents and descents, we know that $\pi_{t_1}<\pi_{t_1+1}$ and $\pi_{t_2}>\pi_{t_2+1}$. There are three cases to consider.
\begin{itemize}
\item[(1)] If $t_2=t_1+1$, it means that there is a subsequence $\pi_{t_1},\pi_{t_2},\pi_{t_2+1}=i,i+1,\pi_{t_2+1}$ such that $i+1>\pi_{t_2+1}$ in $\pi$. It is easy to see that the number of such permutations is $(i-1)(n-2)!$.

\item[(2)] If $t_1=t_2+1$, similarly there is a subsequence $\pi_{t_2},\pi_{t_1},\pi_{t_1+1}=i+1,i,\pi_{t_1+1}$ such that $i<\pi_{t_1+1}$ in $\pi$. It is clearly that the number of such permutations is $(n-2)!+(n-i-1)(n-2)!=(n-i)(n-2)!$, where $(n-2)!$ counts the number of permutations such that $\pi_{t_1+1}=n+1$.
\item[(3)] For $|t_1-t_2|>1$, there are two subcases to consider.
\begin{itemize}
    \item[(3.1)] if $\pi_{t_1+1}\leq n$, the number of such permutations is $(n-i-1)(i-1)(n-2)!$.
   \item[(3.2)] if $\pi_{t_1+1}=n+1$, the number of such permutations is $(i-1)(n-2)!$.
       Therefore, if $i$  is given, there are totally $(n-i)(i-1)(n-2)!$ such permutations in $\mathfrak{S}_n$.
\end{itemize}
\end{itemize}
So, the number of corners in the set of permutation tableaux of length $n\geq 2$ is
\begin{align*}
c(\mathcal{PT}_n)&=\sum_{i=1}^{n-1}|B_i|\\
&=\sum_{i=1}^{n-1}[(i-1)(n-2)!+(n-i)(n-2)!+(n-i)(i-1)(n-2)!]\\
&=(n-2)!\sum_{i=1}^{n-1}[(i-1)+(n-i)+(n-i)(i-1)]\\
&=(n-2)!\sum_{i=1}^{n-1}[(n+1)i-i^2-1]\\
&=(n-1)!\times\frac{n^2+4n-6}{6}.\\
\end{align*}
This completes the proof.
\end{proof}

On the basis of  Theorem \ref{corner-number-A}, we can  give an affirmative answer to Conjecture \ref{conj1}.
\begin{thm}\label{corner-number-TT}
For positive integer $n$, the number of corners in the set of tree-like tableaux of size $n$ is $n!\times \frac{n+4}{6}$ for $n\geq 2$ and $1$ for $n=1$, i.e.,
\[c(\mathcal{T}_n)=\left\{\begin{array}{ll}
1,&\mathrm{~if~}n=1,\\
n!\times \frac{n+4}{6},&\mathrm{~if~}n\geq 2.\end{array}\right.\]
\end{thm}

\begin{proof}
For $n=1$, it is easy to verify that the result is right.
Now we consider the case $n\geq 2$.
Recall that $X_n$ is the  subset of $\mathcal{PT}_n$ such that for any permutation tableau $PT$ in $X_n$, the border edge labeled with $n$ of the underlying diagram of $PT$ is a south step.
By Lemma \ref{lem:bij-PT-Per}, we know that for a given permutation tableau $PT$,  the border edge labeled with $n$ is a south step in the underlying diagram of $PT$ if and only if $n$ is an ascent in $\varphi(PT)$, which means that $\varphi(PT)$ fixes $n$. Evidently, there are $(n-1)!$ such permutations in $\mathfrak{S}_n$, which implies that
$|X_n|=(n-1)!.$

Hence the number of corners in $\mathcal{T}_n$ is
\begin{align*}
c(\mathcal{T}_n)&=c(\mathcal{PT}_n)+|X_n|\\
&=(n-1)!\times\frac{n^2+4n-6}{6}+(n-1)!\\
&=(n-1)!\times\frac{n^2+4n}{6}\\
&=n!\times \frac{n+4}{6}.
\end{align*}
\end{proof}

As L.-Z. \cite{LZ} gave the number of occupied corners, see Proposition \ref{prop-Z}, we can give an enumerative result for \emph{non-occupied corners} in $\mathcal{T}_n$.
\begin{cor}\label{cor:noc}The number of non-occupied corners in $\mathcal{T}_n$ is $n!\times\frac{n-2}{6} $ for $n\geq 3$ and zero for $n=1,2$.
\end{cor}

\section{Further Results about Corners}\label{sec:further_results}

In this section, we give an alternative proof of the enumeration of corners,
by constructing a bijection between corners in tree-like tableaux
and ascending runs of size 1 in permutations.
This answers to a question raised in \cite{GGSY2}.
Then, we present a conjecture about a possible refinement
of the enumeration of corners, by taking in account
the two statistics over tree-like tableaux, $top$ and $left$ defined in \cite{ABN}.
They count the number of non-root points in the first row
and in the first column respectively.
We end this section by explaining the consequences of this conjecture
on the PASEP in the case
$q=1$ and $\gamma=\delta=0$.

\subsection{Bijection between Corners and Ascending Runs of Size 1}\label{subsec-3.1}

An \emph{ascending run} of length $r$ of a permutation $\s=\s_1\cdots\s_n$,
is a sequence $(\s_m,\ldots,\s_{m+r-1})$ such that
$$\s_{m-1}>\s_m<\s_{m+1}<\cdots<\s_{m+r-2}<\s_{m+r-1}>\s_{m+r},$$
with the convention that $\s_{0}=n+1$ and $\s_{n+1}=0$.
In particular, $\s$ has a run of size 1 means that there exists $i\in[n]$
such that $\s_{i-1}>\s_{i}>\s_{i+1}$.

In order to build the bijection, we need a preliminary result about
\emph{non-ambiguous trees} \cite{ABB}, which are shorted for $nats$(i.e., a $nat$ represents a non-ambiguous tree). They correspond to
rectangular shaped tree-like tableaux.
The \emph{height} (resp. \emph{width}) of a non-ambiguous tree
is its number of row (resp. column) minus 1.
We have the following result about these objects (it is a reformulation of
\cite[Proposition~1.16]{ABDOHLZ}).
\begin{prop}\label{prop:nat}
Non-ambiguous trees of height $h$ and width $w$ are in bijection
with permutations $\sigma$ of $\{\b{1},\b{2},\ldots,\b{w},
\r{0},\r{1},\ldots,\r{h}\}$,
finishing by a pointed element and such that
if two consecutive elements $\s_i$ and $\s_{i+1}$ are both pointed or not pointed,
then $\s_i<\s_{i+1}$.
\end{prop}
\begin{proof}
The initial result is that, non-ambiguous trees of height $h$ and width $w$
are in bijection with pairs $(u,v)$ of 2-colored words,
with blue letters on $[w]$ and red letters on $[h]$,
where each letter appear exactly once (in $u$ or in $v$),
letters in blocks of the same colors are decreasing,
$u$ (resp. $v$) ends by a red (resp. blue) letter.

In order to obtain the Proposition~\ref{prop:nat},
we turn pairs $(u,v)$ into the desired permutations $\sigma$.
Let us consider a non-ambiguous tree $nat$ and its corresponding pair
$(u,v)$. We start by constructing a pair $(u',v')$ by replacing
the blue (resp. red) letters $i$ of $u$ and $v$
by the uncolored (resp. uncolored pointed) letters $w-i+1$ (resp. $h-i+1$).
The permutation $\sigma$ corresponding to $nat$ is $v'\r{0}u'$.
\end{proof}

The bijection between corners and ascending runs of size 1
is decomposed into two steps: Lemma~\ref{lemme:coupe} and
Lemma~\ref{lemme:triplet_to_permutation}.

\begin{lem}\label{lemme:coupe}
For $n\geq1$, there is a bijection between corners in $\T_n$ and triplets $(T_l,T_r,nat)$ such that
\begin{itemize}
\item $T_l$ is a tree-like tableau of size $n_l$,
\item $T_r$ is a tree-like tableau of size $n_r$,
\item $n_l+n_r+1=n$,
\item $nat$ is a non-ambiguous tree of height $left(T_r)+1$ and width $top(T_l)+1$.
\end{itemize}
\end{lem}
\begin{proof}
We use two times the idea exposed in Lemma 4 of \cite{ABB}.
Let $T$ be in $\T_n$ and $c$ one of its corners. We start by cutting $T$
along the lines corresponding to the bottom and the right edges of $c$,
as shown in Figure~\ref{fig:coupe_a}. We denote by $L$ the bottom part,
$M$ the middle part and $R$ the right part. In order to obtain
$T_l$ we add to $L$ a first row whose length is equal to the number of columns of $M$
minus 1. There is exactly one way to add dots in this first row
for $T_l$ to be a tree-like tableau: we put them inside the cells
corresponding to non empty columns in $M$. In a similar way, we obtain $T_r$ from $R$
by adding a dotted first column. $T_l$ and $T_r$ are two tree-like tableaux,
and the sum of their length is equal to the length of $T$, hence $n_l+n_r+1=n$.
Finally, removing the empty rows and columns of $M$ we obtain $nat$.
This procedure is illustrated in Figure~\ref{fig:coupe_b}.
It should be clear that the construction can be reversed and that
$T_l$, $T_r$ and $nat$ verifies the desired conditions.
\begin{figure}[h!]
    \begin{center}
        \captionsetup[subfigure]{width=150pt}
        \subfloat[The cutting of $T$ defined by the corner $c$.]{\label{fig:coupe_a}\includegraphics[scale=.5]{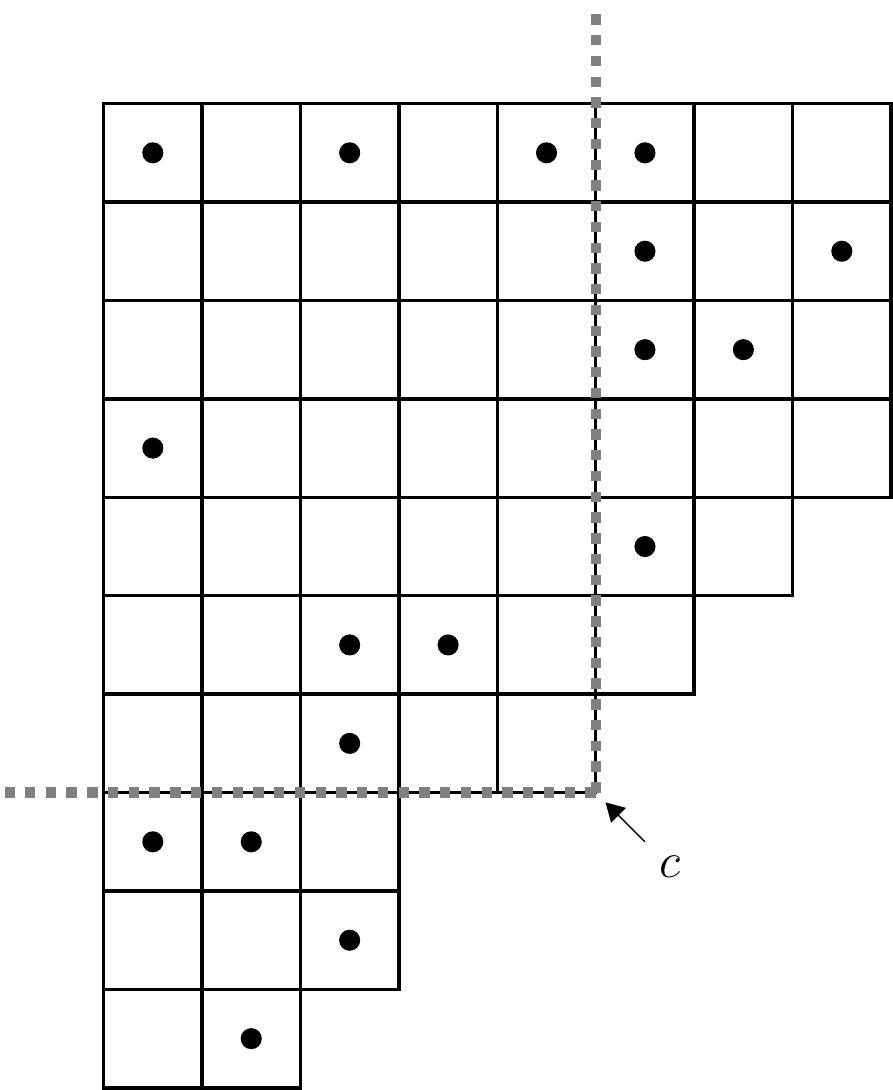}}
        \hfil\hfil\hfil\hfil
        \subfloat[The three parts obtained from $T$. (To obtain $nat$, the gray cells should be removed.)]{\label{fig:coupe_b}\includegraphics[scale=.5]{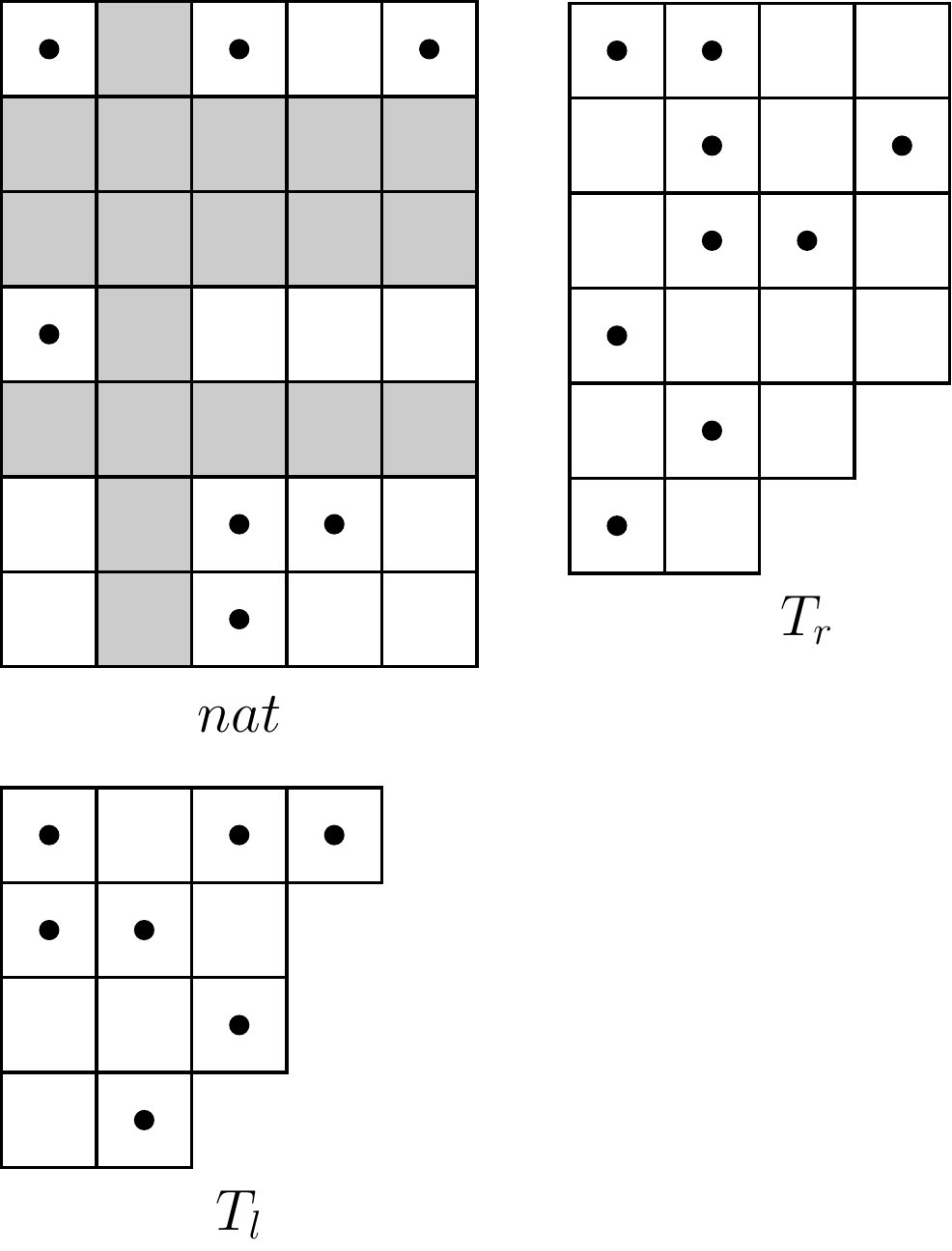}}

        \caption{An example of the bijection of Lemma~\ref{lemme:coupe}.}
        \label{fig:coupe}
    \end{center}
\end{figure}
\end{proof}

\begin{lem}\label{lemme:triplet_to_permutation}
The triplets $(T_l,T_r,nat)$ satisfying all the conditions in Lemma \ref{lemme:coupe} are in bijection with runs of size 1 in permutations of size $n$.
\end{lem}
\begin{proof}
The idea of the proof is the following: from a triplet $(T_l,T_r,nat)$ we construct
a permutation $\s$ of size $n$ which have a run of size 1 $(\s_k)$ such that $\s_k=n_l+1$ for some $k$.
$T_l$ gives the ordering of the values smaller than $\s_k$,
$T_r$ the ordering of the values bigger than $\s_k$,
and $nat$ tells us how we mix them and where we put $\s_k$.

Tree-like tableaux of size $n$ with $k$ points in the first column are in bijection with
permutations of size $n$ with exactly $k$ cycles.
Indeed, by \cite[Proposition~1.3]{ABN}, they
are in bijection with permutation tableaux of size $n$ with $k$
unrestricted rows. Moreover, by \cite{CN}[Theorem~1], they are in bijection with
permutations of size $n$ with $k$ right-to-left minimum.
Finally, the statistic of right-to-left minimum is equi-distributed
with the statistic left-to-right maximum, which is itself equi-distributed
with the statistic of the number of cycles as shown by the
``transformation fondamentale" of Foata-Sch\"utzenberger (\cite[Proposition~1.3.1]{S}
or \cite[Section~1.3]{FS}). Using an axial symmetry with respect to the main diagonal
of the underlying diagram of $T$, we deduce the same result for tree-like tableaux
of size $n$ with $k$ points in the first row.
We denote by $h$ and $w$ the height and the width of $nat$ respectively.
Let $l\s$ (resp. $r\s$) be the permutation associated to $T_l$ (resp. $T_r$)
by this bijection. We denote by $L_1,\cdots,L_w$ (resp. $R_1,\cdots,R_h$)
the disjoint cycles of $l\s$ (resp. $r\s$), such that if $i<j$, then
the maximum of $L_i$ (resp. $R_i$) is smaller than the maximum of $L_j$ (resp.$R_j$).
In addition, we shift the values of the $R_i$ by $n_l+1$.
We will write a cycle without parenthesis and with its biggest element
at the first position.
For example, suppose $l\s=(6)(7523)(9184)$ and $r\s=(423)(5)(716)(98)$, then
$L_1=6$, $L_2=7\text{ }5\text{ }2\text{ }3$, $L_3=9\text{ }1\text{ }8\text{ }4$,
$R_1=14\text{ }12\text{ }13$, $R_2=15$, $R_3=17\text{ }11\text{ }16$ and $R_4=19\text{ }18$.
Let $m$ be the word corresponding to $nat$.
If $\r{0}$ is not the last letter of $m$ and if the letter after $\r{0}$ is not pointed,
we can uniquely represent $m$ as $$m=u\r{0}\b{a_1}\r{b_1}\b{a_2}\r{b_2}\cdots\b{a_p}\r{b_p},$$
where $u$ can be empty,
and the words $\r{b_i}$ (resp. $\b{a_i}$) consist of a non empty increasing sequence
of pointed (resp. non pointed) letter.
In this case, we replace $m$ by swaping
subwords $\b{a_i}$ and $\r{b_i}$ for all $i$ ($1\leq i\leq p$), i.e.,
$$m^*=u\r{0}\r{b_1}\b{a_1}\r{b_2}\b{a_2}\cdots\r{b_p}\b{a_p}.$$
Obviously, this operation is bijective.

We finish the general construction by substituting
$n_l+1$ for $\r{0}$, $L_i$ for $\r{i}$ and $R_i$ for $\b{i}$.
For example, if $$m=\b{2}\b{3}\r{2}\r{3}\b{1}\b{4}\r{0}\r{1}$$ then we obtain the run
of size one $$15\text{ }17\text{ }11\text{ }16\text{ }7\text{ }5\text{ }2\text{ }3\text{ }9\text{ }1\text{ }8\text{ }4\text{ }14\text{ }12\text{ }13\text{ }19\text{ }18\text{ }\underline{10}\text{ }6.$$
Another example, if $$m=\r{1}\b{4}\r{0}\b{12}\r{2}\b{3}\r{3}$$ then
$$m^*=\r{1}\b{4}\r{0}\r{2}\b{12}\r{3}\b{3}$$ and thus we get
$$6\text{ }19\text{ }18\text{ }\underline{10}\text{ }7\text{ }5\text{ }2\text{ }3\text{ }14\text{ }12\text{ }13\text{ }15\text{ }9\text{ }1\text{ }8\text{ }4\text{ }17\text{ }11\text{ }16.$$
In order to reverse the construction from a run of size one $(\s_k)$,
we use the ``transformation fondamentale" in each maximal sequence of integers
smaller (resp. larger) than $\s_k$. This way, we are able to identify
the $L_i$ (resp. $R_i$), so that we can recover $l\s$, $r\s$ and $nat$.
For example, if we consider the run
$$\text{4 2 6 11 9 12 \underline{8} 3 7 1 5 10},$$
we obtain
$$
\underbrace{4\text{ }2}_{L_\r{2}}
\underbrace{6}_{L_\r{3}}
\underbrace{11\text{ }9}_{R_\b{2}}
\underbrace{12}_{R_\b{3}}
\underbrace{8}_{\ _\r{0}}
\underbrace{3}_{L_\r{1}}
\underbrace{7\text{ }1\text{ }5}_{L_\r{4}}
\underbrace{10}_{R_\b{1}},$$
hence
$$
l\s = (3)(4\text{ }2)(6)(7\text{ }1\text{ }5)
\text{, }
r\s = (2)(3\text{ }1)(4)
\text{, }
m^*=\r{2}\r{3}\b{2}\b{3}\r{0}\r{1}\r{4}\b{1},
m = \r{2}\r{3}\b{2}\b{3}\r{0}\b{1}\r{1}\r{4}
.$$
\end{proof}
As a consequence of Lemma~\ref{lemme:coupe} and Lemma~\ref{lemme:triplet_to_permutation},
we have the following theorem.
\begin{thm}\label{thm:corners_runs}
For $n\geq1$, corners in $\T_n$ are in bijection with runs of size 1 in $\mathfrak{S}_n$.
\end{thm}
Even if we send corners to runs of size 1, the two statistics does not have the same distribution.
For example, the permutation 321 have 3 runs of size 1 while a tree-like tableau of size 3 cannot have 3 corners.

We deduce from Theorem~\ref{thm:corners_runs} the enumeration of corners,
by enumerating ascending runs of size 1.
\begin{cor}
The number of corners in $\T_n$ is $n!\cdot\frac{n+4}{6}$ for $n\geqslant2$ and 1 for $n=1$.
\end{cor}
\begin{proof}
Theorem~\ref{thm:corners_runs} implies that enumerating
ascending runs of size 1 in permutations is sufficient.
For $n=1$, there is a unique ascending run which corresponds to the unique
permutation of size one. For $n\geqslant2$,
we start by considering the ascending runs of size 1 which are at the first position.
Such a run happens in a permutation $\s$, if $\s_1>\s_2$.
We have $\binom{n}{2}$ to choose these two integers, and then,
$(n-2)!$ ways to construct the permutation. We have the same result for ascending
runs of size 1 which are at the last position. The remaining runs happen if there exists
$2\leq i \leq n-1$ such that $\s_{i-1}>\s_i>\s_{i+1}$. We have $\binom{n}{3}$ ways
to choose these three integers, and then, $(n-2)!$ ways to construct the permutation.
In total, the number of runs of size 1 is
$$\left[2\cdot\binom{n}{2}+\binom{n}{3}\right]\cdot(n-2)!=n!\cdot\frac{n+4}{6}$$
\end{proof}

\subsection{($a$,$b$)-analogue of the Average Number of Non-occupied Corners}

As explained in \cite[Section 4]{LZ}, computing the average number of corners
gives us the average number of locations where a particle may jump to the left
or to the right in the PASEP model, in the case $\alpha=\beta=q=1$ and $\delta=\gamma=0$.
Computing the ($a$,$b$)-analogue of average number of corners
$$c_n(a,b):=\sum_{T\in\T_n}c(T) \cdot w(T),$$
where $w(T) = a^{top(T)}b^{left(T)},$
would extend the result to the case $q=1$ and $\delta=\gamma=0$,
if we replace $a$ by $\alpha^{-1}$ and $b$ by $\beta^{-1}$.

The ($a$,$b$)-analogue of the average number of tree-like tableaux,
was computed in \cite{ABN}, it is equal to
$$T_n(a,b):=\sum_{T\in\T_n} w(T) = (a+b)(a+b+1)\cdots(a+b+n-2).$$
It turns out that the ($a$,$b$)-analogue of the average number of occupied corners
is also $T_n(a,b)$. In order to prove this, we just need to redo the short proof
of \cite[Section~3.2]{LZ} keeping track of left and top points.
As a consequence of this result, computing the ($a$,$b$)-analogue
for corners or for non-occupied corners, is equivalent.
In this section, we focus on non-occupied corners,
because their study seems easier.
We denote by $noc_n(a,b)$ the ($a$,$b$)-analogue of the average number
of non-occupied corners, i.e.,
$$noc_n(a,b):=\sum_{T\in\T_n}noc(T) w(T),$$
where $noc(T)$ is the number of non-occupied corners of $T$.
In particular, Corollary~\ref{cor:noc} implies that $noc_n(1,1)=n!\cdot\frac{n-2}{6}$.
Using an implementation of tree-like tableaux in Sage \cite{sage},
the following conjecture has been experimentally confirmed until $n=10$.

\begin{conj}
For $n\geq 3$, the ($a$,$b$)-analogue of the enumeration of non-occupied corners is
$$
noc_n(a,b)
=
\left(
    (n-2)ab+\binom{n-2}{2}(a+b)+\binom{n-2}{3}
\right)
\cdot
T_{n-2}(a,b)
$$
\end{conj}
In order to obtain the conjecture about corners,
we just have to add $T_n(a,b)$ to $noc_n(a,b)$. So, $c_n(a,b)$ can be rewritten as follows:
$$
c_n(a,b)
=
\left(
    a^2 + b^2 + nab + \frac{(n^2-n-4)(a+b)}{2} + \frac{(n+2)(n-2)(n-3)}{6}
\right)
\cdot
T_{n-2}(a,b).
$$
Let $X(s)$ be the random variable counting the number of locations of a state $s$ of size $n$
of the PASEP, where a particle may jump to the right or to the left.
We can compute the conjectural expected value of $X$ by using the formula of \cite[Section 4]{LZ},
\begin{align*}
\mathbb{E}(X)   &= \frac{1}{T_{n+1}(a,b)}\sum_{T\in\T_{n+1}}w(T)(2c(T)-1). \\
                &= \frac{
                    2\cdot(
                        a^2 + b^2 + (n+1)ab +
                        \frac{(n^2+n-4)(a+b)}{2} +
                        \frac{(n+3)(n-1)(n-2)}{6}
                    )
                }{
                    (a+b+n-1)(a+b+n-2)
                } -1\\
                &=\frac{
                    6[a^2 + b^2 + (n+1)ab] + 3(n^2+n-4)(a+b) + (n+3)(n-1)(n-2)              }{                   3(a+b+n-1)(a+b+n-2)
                } -1\\
                &= \frac{
                    3(a^2 + b^2) + 6nab + 3(n^2 - n - 1)(a + b) + n(n-1)(n-2)
                }{
                    (a+b+n-1)(a+b+n-2)
                }.
\end{align*}

Instead of studying $noc_n(a,b)$ as a sum over tree-like tableaux, we will study it
as a sum over non-occupied corners in $\T_n$. Let $noc$ be a non-occupied corner
of a tree-like tableau $T$, we define the weight of $noc$ as
$$w(noc):=w(T).$$
Let $NOC(\T_n)$ be the set of non-occupied corners in $\T_n$, we can rewrite
$noc_n(a,b)$ as
$$noc_n(a,b)=\sum_{noc\in NOC(\T_n)}w(noc).$$
The study of this conjecture brings a partitioning of non-occupied corners.
We denote by $NOC_{a,b}(\T_n)$ the set of non-occupied corners
with no point above them, in the same column, except in the first row
and no point at their left, in the same row, except in the first column.
The set of non-occupied corners with no point above them, except in the first row,
or no point at their left, except in the first column, but not both in the same time,
are denoted by $NOC_{a,1}(\T_n)$ and $NOC_{1,b}(\T_n)$ respectively.
The remaining corners are regrouped in $NOC_{1,1}(\T_n)$.
The different types of non-occupied corners are illustrated in Figure~\ref{fig:noc_partitioning}.
\begin{figure}[h!]
    \begin{center}
        \captionsetup[subfigure]{width=100pt}
        \subfloat[$NOC_{a,b}(\T_n)$]{
            \label{fig:noc_ab}
            \includegraphics[scale=.4]{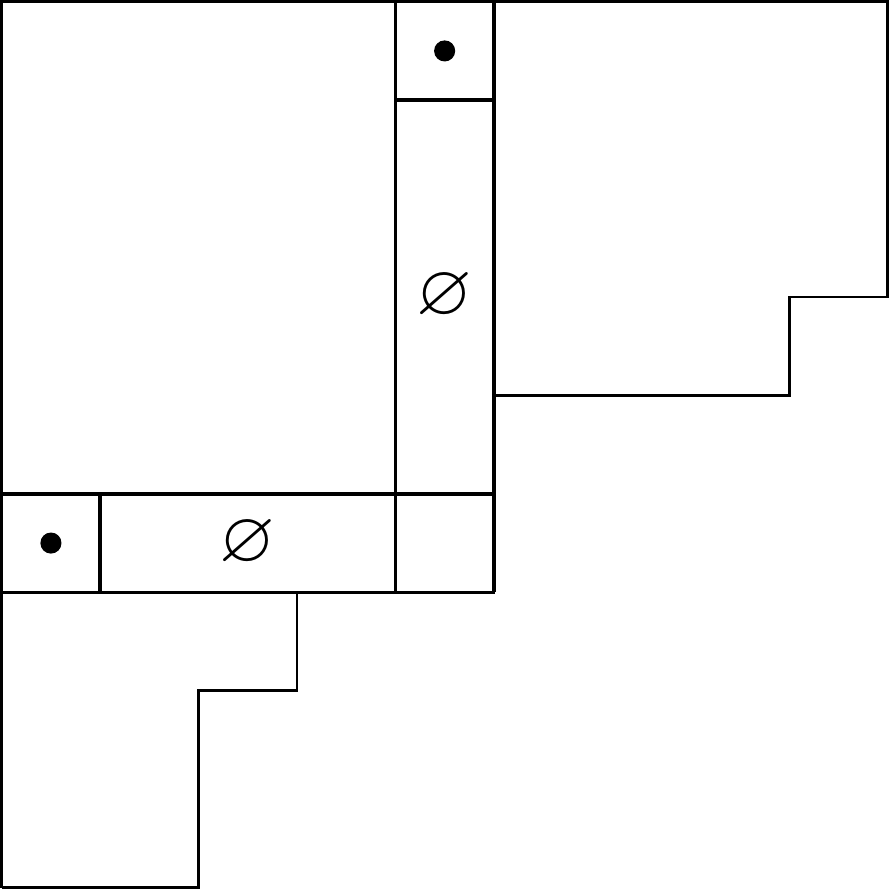}
        }
        \hfil\hfil
        \subfloat[$NOC_{a,1}(\T_n)$]{
            \label{fig:noc_a1}
            \includegraphics[scale=.4]{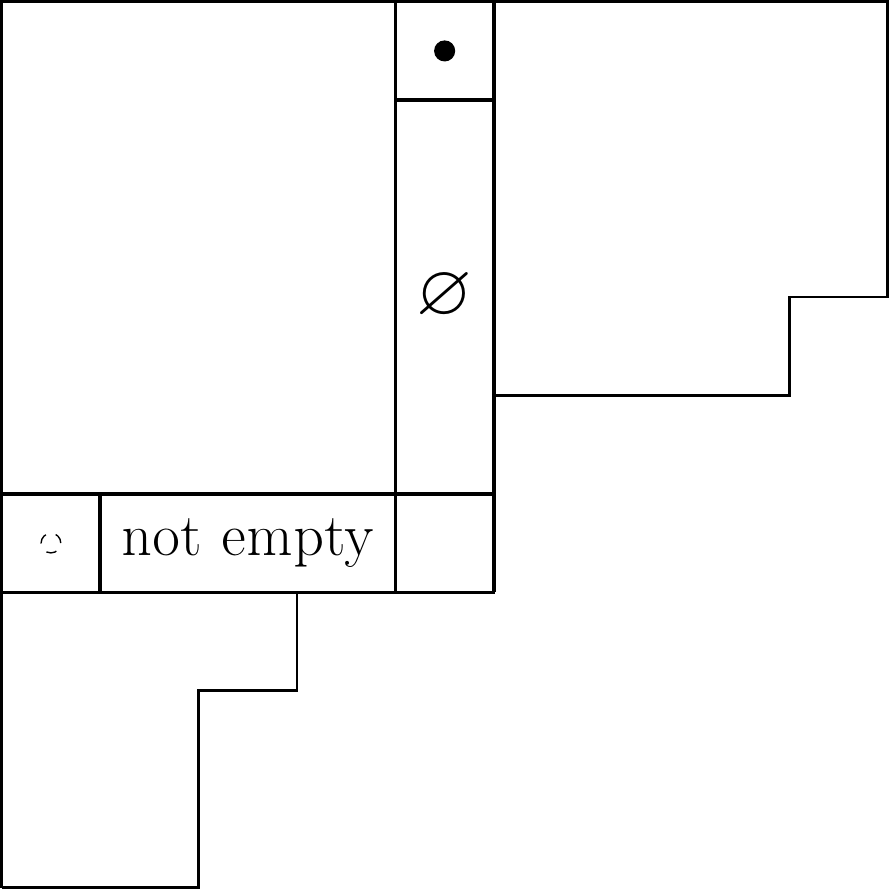}
        }
        \hfil\hfil
        \subfloat[$NOC_{1,b}(\T_n)$]{
            \label{fig:noc_1b}
            \includegraphics[scale=.4]{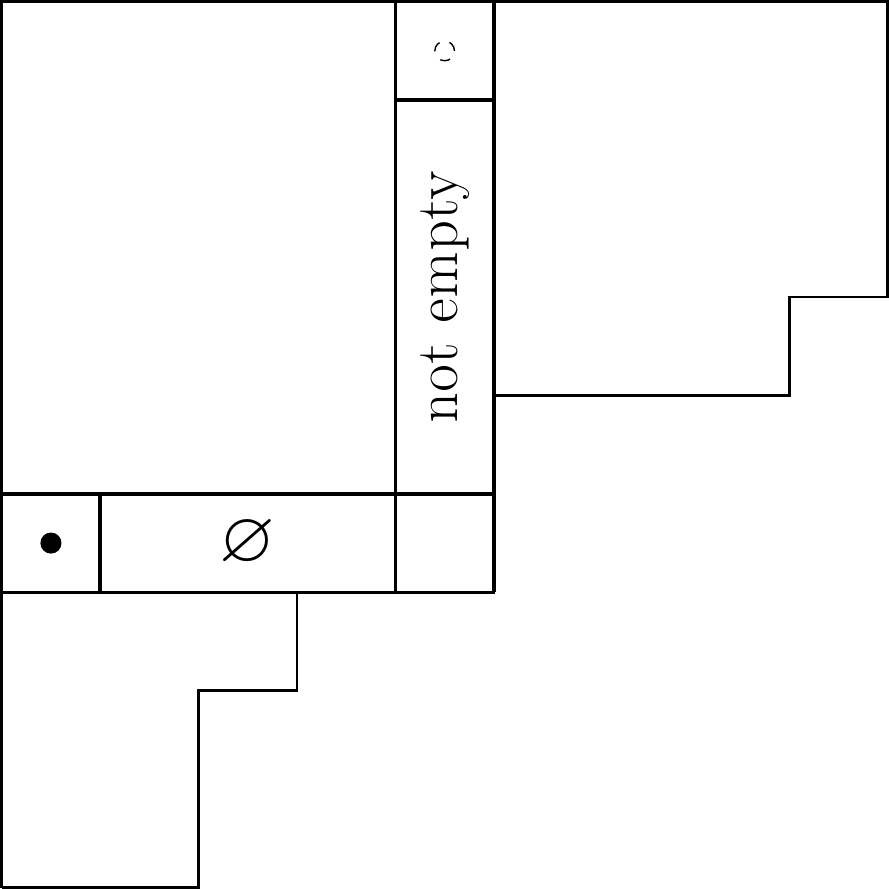}
        }
        \hfil\hfil
        \subfloat[$NOC_{1,1}(\T_n)$]{
            \label{fig:noc_11}
            \includegraphics[scale=.4]{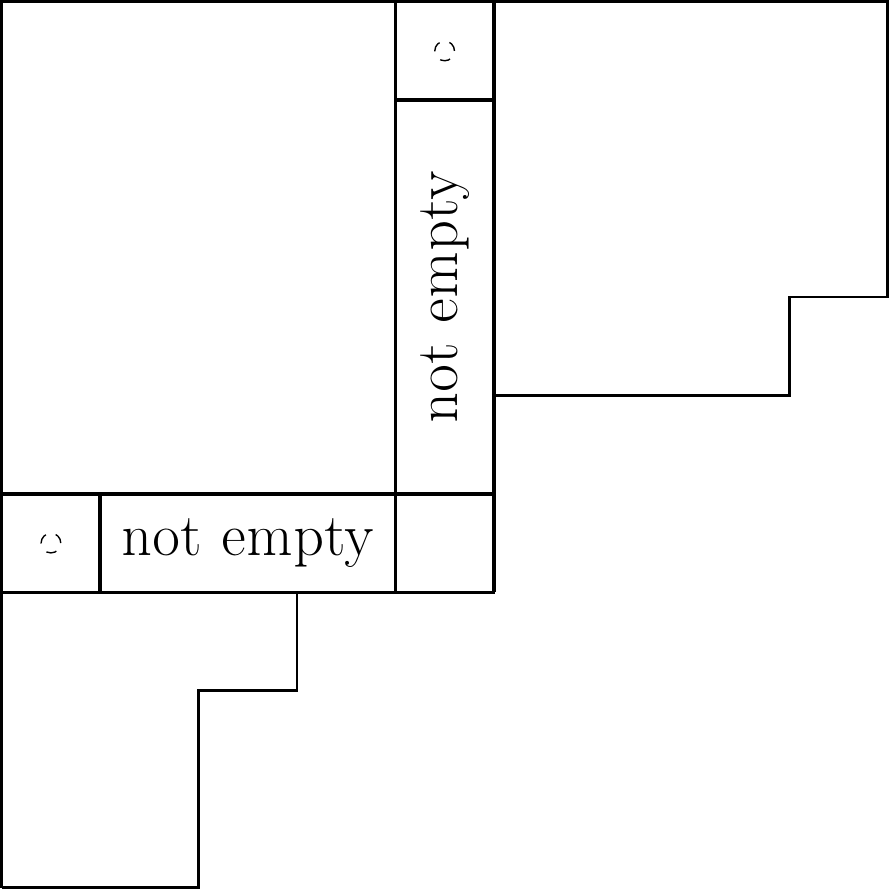}
        }
    \end{center}
    \caption{
        Partitioning of non-occupied corners in $\T_n$.
        ({\protect\tikz \protect\draw[dashed,color=black] (0,0) circle (.1);}
        means that the cell can be whether empty or occupied.)
    }
    \label{fig:noc_partitioning}
\end{figure}

\begin{prop}
For $n\geq 3$, the ($a,b$)-analogue of the enumeration of $NOC_{a,b}(\T_n)$ is
$$
\sum_{noc\in NOC_{a,b}(\T_n)} w(noc)
=
(n-2) \cdot ab \cdot T_{n-2}(a,b)
.$$
\end{prop}
\begin{proof}
In order to show that, we put in bijection
non-occupied corners of $\T_n$ of this shape and
the set of pairs $(T,i)$ where $T$ is a tree-like tableau
of size $n-2$ and $i$ is an interstice between
two consecutive border edges of $T$.
Let $noc\in NOC_{a,b}(\T_n)$ and $T'$ be its tree-like tableau of size $n$.
Let $j$ be the integer such that $noc$ is the cell $(j,j+1)$,
We obtain a tree-like tableau $T$ of size $n-2$ by removing
the row $j$ and the column $j+1$.
The north-west corner $i$ of $noc$ corresponds to an interstice between two consecutive
border edges of $T$, we associate to $noc$ the pair $(T,i)$.
Conversely, let us consider a pair $(T,i)$.
We construct a tree-like tableau $T'$ as follows,
we add to $T$ a row and a column ending a common cell $c$
with $i$ as its north-west corner, and we had a point to the left-most (resp. highest)
cell of the new row (resp. column). In particular, $c$ is in $NOC_{a,b}(\T_n)$.
The bijection is illustrated in Figure \ref{fig:noc_ab_bij}.
\begin{figure}[h!]
    \begin{center}
        \captionsetup[subfigure]{width=150pt}
        \subfloat[Bijection between $NOC_{a,b}(\T_n)$ and pairs $(T,i)$.]
            {
                \label{fig:noc_ab_bij}
                \includegraphics[scale=.4]{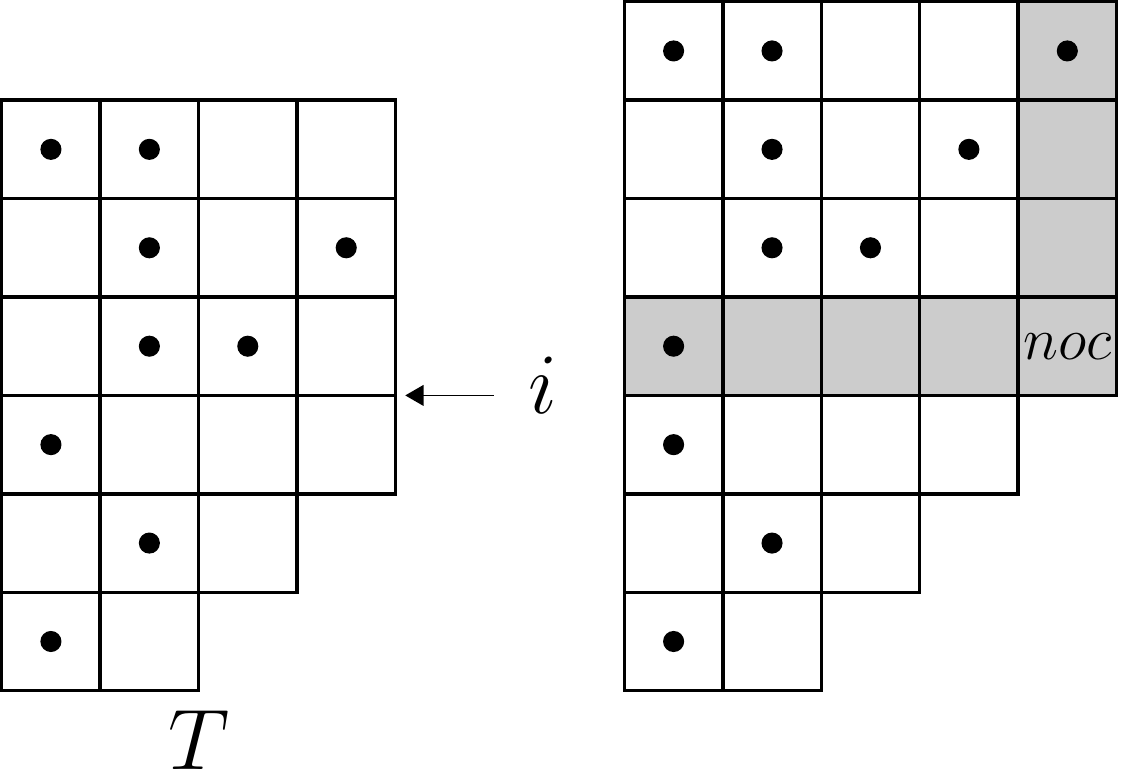}
            }
        \hfil\hfil
        \subfloat[
            Bijection between pairs $(T,r)$
            and $NOC_{a,b}(\T_n)\bigcup NOC_{a,1}(\T_n)$ .
        ]
            {
                \label{fig:noc_a1_bij}
                \includegraphics[scale=.4]{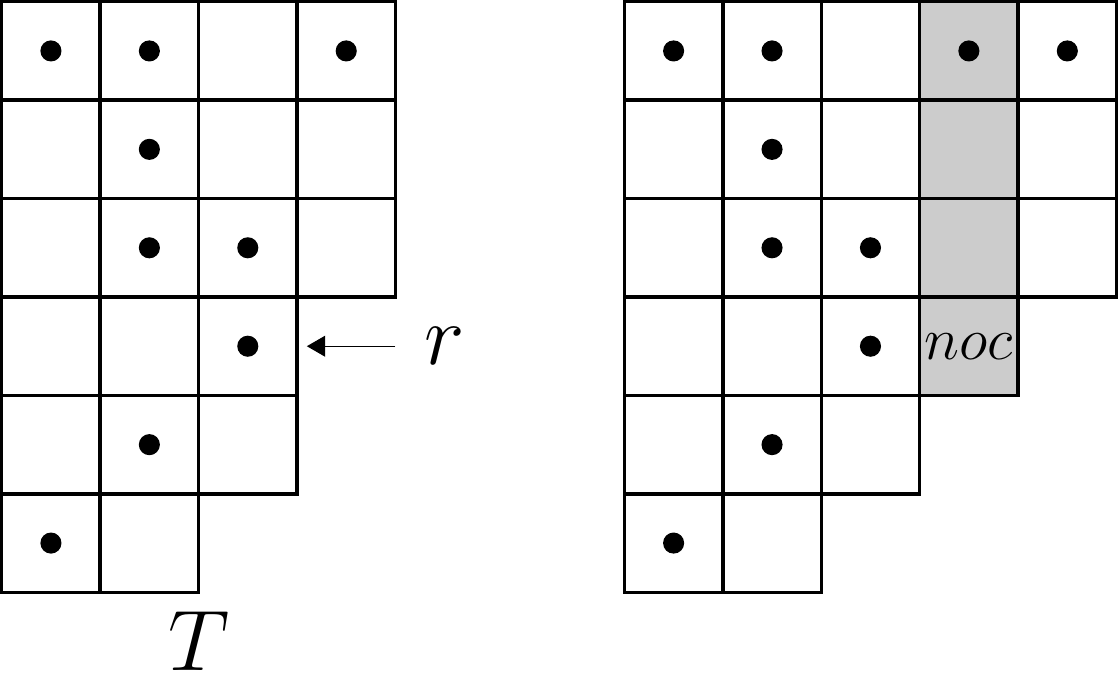}
            }
        \caption{}
        \label{fig:noc_a1_bijection}
    \end{center}
\end{figure}

For each tree-like tableau $T$ of size $n-2$, there are $n-2$ choices
of interstice $i$, in addition, the weight of $noc$ is equal to $ab\cdot w(T)$.
As a result,
$$
\sum_{
    noc\in NOC_{a,b}(\T_n)
} w(noc)
=
\sum_{
    \substack{
        T\in\T_{n-2} \\
        i \text{ interstice of } T
    }
} ab \cdot w(T)
=
(n-2) \cdot ab \sum_{T\in\T_{n-2}} w(T)
=
(n-2) \cdot ab \cdot T_{n-2}(a,b)
.$$
\end{proof}
We are also able to give an ($a,b$)-analogue of the enumeration of $NOC_{a,1}(\T_n)$
and $NOC_{1,b}(\T_n)$. In order to do that, we need an ($a,b$)-analogue
of the enumeration of tree-like tableaux of size $n$
with a fixed number of rows $k$,
$$
\A{n,k}
=
\sum_{
    \substack{T\in\T_n \\ T\text{ has }k\text{ rows }}
}
a^{top(T)}b^{left(T)}.
$$

From \cite[Proposition~3.4]{ABN} we already know
that $A_{1,1}(n,k)$ is the Eulerian number and satisfies
$$A_{1,1}(n+1,k) = kA_{1,1}(n,k) + (n+2-k)A_{1,1}(n,k-1).$$
In the general case, the linear recurrence satisfied by $\A{n,k}$ is
\begin{equation}\label{Euler_a_b}
\A{n+1,k}=(a-1+k)\A{n,k}+(b+n+1-k)\A{n,k-1}.
\end{equation}
As in \cite{ABN}, we consider the Eulerian polynomial:
$$A_n(t):=\sum_{k=1}^{n}\A{n,k}t^k.$$
\begin{lem}
For $n\geqslant 2$,
$$A_n(1) = T_n(a,b)\text{ and } A_n'(1) =(a+bn+\binom{n}{2}-1)T_{n-1}(a,b).$$
\end{lem}
\begin{proof}
The first identity is a consequence of the definition of $A_n(t)$.
For the second one,
we deduce from \eqref{Euler_a_b} that $A_n(t)$ satisfies the recurrence relation
\begin{equation*}
\begin{split}
A_n(t)&=(a-1)A_{n-1}(t)+tA'_{n-1}(t)+(b+n-1)tA_{n-1}(t)-t^2A'_{n-1}(t)\\
      &=(a-1)A_{n-1}(t)+(b+n-1)tA_{n-1}(t)+t(1-t)A'_{n-1}(t),
\end{split}
\end{equation*}
with initial condition $A_1(t)=t$.
Hence, by differentiating and evaluating at $t=1$,
we get the following recurrence relation for $A'_n(1)$
\begin{align*}
A_n'(1) &= (a-1)A'_{n-1}(1) + (b+n-1)(A'_{n-1}(1)+A_{n-1}(1)) - A'_{n-1}(1)\\
        &= (a+b+n-3)A'_{n-1}(1)+(b+n-1)A_{n-1}(1).
\end{align*}
For $n\geqslant 3$, dividing by $T_{n-1}(a,b)$, we get
$$\frac{A_n'(1)}{T_{n-1}(a,b)}=\frac{A'_{n-1}(1)}{T_{n-2}(a,b)}+(b+n-1).$$
Since $A'_2(1)=a+2b$, for $n\geqslant 2$,
$$A'_n(1)=(a+bn+\binom{n}{2}-1)(a+b+n-3)\cdots(a+b).$$
\end{proof}
We can now prove the following result,
\begin{prop}
For $n\geq 3$, the ($a,b$)-analogue of the enumeration of $NOC_{a,1}(\T_n)$ is
$$
f_n(a,b)
:=
\sum_{noc\in NOC_{a,1}(\T_n)} w(noc)
=
\binom{n-2}{2}\cdot a \cdot T_{n-2}(a,b)
,$$
and the ($a,b$)-analogue of the enumeration of $NOC_{1,b}(\T_n)$ is
$$
g_n(a,b)
:=
\sum_{noc\in NOC_{1,b}(\T_n)} w(noc)
=
\binom{n-2}{2}\cdot b \cdot T_{n-2}(a,b).$$
\end{prop}
\begin{proof}
In order to compute $f_n(a,b)$ we put in bijection elements of
$NOC_{a,1}(\T_n)\bigcup NOC_{a,b}(\T_n)$
with pairs $(T,r)$ where $T$ is a tree-like tableaux of size $n-1$ and $r$
is a row of $T$, different from the first one.
Let $noc\in NOC_{a,1}(\T_n)$ and $T'$ its tree-like tableau
of size $n$. Let $j$ be the integer such that $noc$ is the cell $(j,j+1)$.
We obtain $T$ by removing the column $j+1$, $r$ corresponds to the row $j$.
It should be clear that this operation is revertible.
The bijection is illustrated in Figure~\ref{fig:noc_a1_bij}.
Using the previous notations, $w(noc)=a\cdot w(T)$.
As a result,
\begin{align*}
f_n(a,b)    &= \sum_{
                  \substack{
                      T\in\T_{n-1}\\
                      r\text{ a row of $T$}
                  }
               } a\cdot w(T)
               -
               \sum_{noc\in NOC_{a,b}(\T_n)} w(noc) \\
            &= a \cdot \sum_{k=1}^{n-1} (k-1) \cdot \A{n-1,k}
               -
               (n-2) \cdot ab \cdot T_{n-2}(a,b) \\
            &= a \cdot (A'_{n-1}(1) - A_{n-1}(1)) - (n-2) \cdot ab \cdot T_{n-2}(a,b) \\
            &= a\left[(a+(n-1)b+\binom{n-1}{2}-1)-(a+b+n-3)-(n-2)b\right]\cdot T_{n-2}(a,b) \\
            &=a\binom{n-2}{2} \cdot T_{n-2}(a,b)
\\
\end{align*}

The axial symmetry with respect to the main diagonal
of the underlying diagram gives a bijection between
$NOC_{a,1}(\T_n)$ and $NOC_{1,b}(\T_n)$,
a top point becomes a left point and conversely.
Hence, $g_n(a,b)$ can be deduced from $f_n(a,b)$
by the identity $g_n(a,b)=f_n(b,a)$. Since $T_{n-2}(a,b)$
is a symmetric polynomial,
$g_n(a,b)=b\binom{n-2}{2} \cdot T_{n-2}(a,b)$.
\end{proof}

To prove the conjecture, we miss the ($a,b$)-analogue
of the enumeration of $NOC_{1,1}(\T_n)$.
The main issue is how to link these non-occupied corners with tree-like
tableaux of smaller size.

\section{Conclusion and Remarks}
We consider the enumeration of corners in tree-like tableaux of size $n$ via counting corners in permutation tableaux of length $n$, including non-occupied corners in tree-like tableaux of size $n$, and give exact formulas for these numbers.
Moreover, we gave a bijection between corners and runs of size 1 in permutations,
which gave an alternative proof of the enumeration of corners.
Finally, we introduced an ($a$,$b$)-analogue of this enumeration, and explained the
consequences on the PASEP.

It is worthy to note that the number of non-occupied corners in tree-like tableaux of size $n+1$ occurs in \cite[A005990]{Sloane}, which enumerates the total positive displacement of all letters in all permutations on $[n]$, i.e,
\begin{equation*}
\sum_{\pi\in \mathfrak{S}_n} \sum_{i=1}^n \max\{\pi_i-i,0\},
\end{equation*}
the number of double descents in all permutations of $[n-1]$ and also the sum of the excedances of all permutations of $[n]$. We say that $i$ is a double descent of a permutation $\pi=\pi_1\pi_2\cdots\pi_n$ if $\pi_i>\pi_{i+1}>\pi_{i+2}$, with $1\leqslant i \leqslant n-2$ and an excedance if $\pi_i>i$, with $1\leq i\leq n-1$.
Besides, they are also related to coefficients of Gandhi polynomials, see \cite{T}.
To find the relationship between both of them is also an interesting problem.

\noindent{\bf Acknowledgements.}
This work was supported by the 973 Project, the PCSIRT Project of the Ministry of Education and the National Science Foundation of China.

A part of this research (Section~\ref{sec:further_results}) was driven by computer exploration using the open-source software \texttt{Sage}~\cite{sage} and its algebraic combinatorics features developed by the \texttt{Sage-Combinat} community~\cite{Sage-Combinat}.

\end{document}